\documentclass{article}%
\usepackage{graphicx}
\usepackage{amsmath}
\usepackage{amscd}
\usepackage{amsfonts}
\usepackage{amssymb}

\newtheorem{theorem}{Theorem}[section]

\newtheorem{proposition}[theorem]{Proposition}

\newtheorem{definition}[theorem]{Definition}
\newtheorem{corollary}[theorem]{Corollary}
\newtheorem{remark}[theorem]{Remark}

\newenvironment{proof}[1][Proof]{\textbf{#1.} }{\ \rule{0.5em}{0.5em}}

\newcommand{\be}{\begin{equation}}
\newcommand{\ee}{\end{equation}}
\newcommand{\bes}{\begin{equation*}}
\newcommand{\ees}{\end{equation*}}

\newcommand{\cM}{\mathcal{M}}

\newcommand{\cR}{\mathcal{R}}
\newcommand{\cS}{\mathcal{S}}

\newcommand{\Rp}{\mathbb{R}_+}
\newcommand{\Rpt}{\mathbb{R}_+^2}

\begin{document}

\title{What type of dynamics arise in E$_0$-dilations of commuting quantum Markov processes?}

\author{Orr Moshe Shalit\\
          Department of Mathematics, Technion\\
          Haifa 32000, Israel\\
          email address: orrms@tx.technion.ac.il}
\date{19.11.2007}
\maketitle

\begin{abstract}
{Let $H$ be a separable Hilbert space. Given two strongly commuting CP$_0$-semigroups $\phi$ and $\theta$ on $B(H)$, there
is a Hilbert space $K \supseteq H$ and two (strongly) commuting E$_0$-semigroups $\alpha$ and $\beta$ such that
$$\phi_s \circ \theta_t (P_H A P_H) = P_H \alpha_s \circ \beta_t (A) P_H$$
for all $s,t \geq 0$ and all $A \in B(K)$.

In this note we prove that if $\phi$ is not an automorphism semigroup then $\alpha$ is cocycle conjugate to the minimal $*$-endomorphic dilation of $\phi$, and that if $\phi$ is an automorphism semigroup then $\alpha$ is also an automorphism semigroup. In particular, we conclude that if $\phi$ is not an automorphism semigroup and has a bounded generator (in particular, if $H$ is finite dimensional) then $\alpha$ is a type I E$_0$-semigroup.

\bigskip

}

\end{abstract}

\section{Introduction}
Let $H$ be a separable Hilbert space. A \emph{CP$_0$-semigroup} on $B(H)$ is a family $\phi = \{\phi_t\}_{t\geq0}$ of contractive, normal, unital and completely positive maps on $B(H)$ satisfying the semigroup property
$$\phi_{s+t}(A) = \phi_s (\phi_t(A)) \,\, ,\,\, s,t\geq 0, A\in B(H) ,$$
$$\phi_{0}(A) = A \,\, , \,\,  A\in B(H) ,$$
and the continuity condition
$$\lim_{t\rightarrow t_0} \langle \phi_t(A)h,g\rangle = \langle \phi_{t_0}(A)h,g\rangle \,\, , \,\, A\in B(H), h,g \in H .$$
A CP$_0$-semigroup is sometimes called a \emph{Quantum Markov Processes}, as it may be considered as noncommutative generalization of a Markov processes. A CP$_0$-semigroup is called an \emph{E$_0$-semigroup} if each of
its elements is a $*$-endomorphism.

The simplest E$_0$-semigroups are automorphism semigroups. The rest of the E$_0$-semigroups can be classified into 3 ``types": type I, type II and type III. There is a complete classification of type I E$_0$-semigroups, and it is known that if $\alpha$ is a type I E$_0$-semigroup then there is a $d \in \{1,2,\ldots, \infty\}$ such that $\alpha$ is cocycle conjugate to the CCR flow of index $d$. See \cite{Arv03} for the whole story.

Let $\phi$ be a CP$_0$-semigroup acting on $B(H)$, and let $\alpha$ be an
E$_0$-semigroup acting on $B(K)$, where $K\supseteq H$. We say that $\alpha$ is an \emph{E$_0$-dilation} of
$\phi$ if for all $t \geq 0$ and $A \in B(K)$
\be\label{eq:dilation}
\phi_t(P_H A P_H) = P_H \alpha_t (A) P_H ,
\ee
(here $P_H$ denotes the orthogonal projection of $K$ onto $H$). In the mid 1990's Bhat proved the following result, known today as ``Bhat's Theorem" (see \cite{Bhat1996}):
\begin{theorem}
\emph{{\bf (Bhat).}} Every CP$_0$-semigroup has a unique minimal
E$_0$-dilation.
\end{theorem}

Bhat's Theorem aroused much interest, and one of the reasons was because it opened up a new way of constructing
E$_0$-semigroups. A possible approach could have been this: construct explicitly a tractable CP$_0$-semigroup, (for example a CP$_0$-semigroup on the algebra of $n \times n$ matrices or more generally a CP$_0$-semigroup with a bounded generator), and look at its minimal E$_0$-dilation. It was hoped at the time that the resulting E$_0$-semigroup would turn out to be an E$_0$-semigroup that has not been seen before.

These hopes were soon extinguished by results of Arveson and Powers.
\begin{theorem}\label{thm:Arv}
\emph{{\bf (Arveson, \cite[Theorem 4.8]{Arv99}).}} Let $\phi$ be a CP$_0$-semigroup with a bounded generator. The minimal E$_0$-dilation of $\phi$ is of type I.
\end{theorem}
Independently, Powers proved that the minimal E$_0$-dilation of a CP$_0$-semigroup acting on the algebra $M_n(\mathbb{C})$ of $n \times n$ matrices is of type I (\cite[Theorem 3.10]{Pow99}). Although Powers' result is contained in Arveson's result, it is worth mentioning his paper not only because he reached the result
using completely different methods, but also because that paper contains an independent proof (which seems to have been forgotten) of the existence of an E$_0$-dilation for CP$_0$-semigroups on matrix algebras.

In \cite{Shalit07a} we raised the question whether every \emph{two-parameter} CP$_0$-semigroup
has a (two-parameter) E$_0$-dilation. We obtained a partial result, which for our purposes in this note
can be stated as follows:
\begin{theorem}\label{thm:scedil}
\emph{{\bf (\cite[Theorem 6.6]{Shalit07a}).}} Let $\phi$ and $\theta$ be two strongly commuting CP$_0$-semigroups on $B(H)$, where $H$ is a separable Hilbert space. Then there is a separable Hilbert space $K\supseteq H$ and two commuting
E$_0$-semigroups $\alpha$ and $\beta$ on $B(K)$ such that
$$\phi_s \circ \theta_t (P_H A P_H) = P_H \alpha_s \circ \beta_t (A) P_H$$
for all $s,t \geq 0$ and all $A \in B(K)$.
\end{theorem}
For a definition of \emph{strong commutation} see Section 4 in \cite{Shalit07a}. Let us point out that every pair of CP$_0$-semigroups on $M_n(\mathbb{C})$ that commute do so strongly. Thus we have:
\begin{corollary}\label{cor:M_n(C)}
\emph{{\bf (\cite[Corollary 6.7]{Shalit07a}).}}
Every two-parameter CP$_0$-semigroup on $M_n(\mathbb{C})$ has an E$_0$-dilation.
\end{corollary}
One can also show that these dilations are minimal in an appropriate sense, but we shall not
make use of the minimality of the two-parameter dilations in this note.

The last two results face us against two immediate problems:
\begin{enumerate}
\item Figure out the structure of the E$_0$-dilation of a given two-parameter CP$_0$-semigroup, especially
in the simplest case when the the CP$_0$-semigroup acts on $M_n(\mathbb{C})$.
\item Try to see whether new E$_0$-semigroups (necessarily not of type I) can arise as ``parts" of the E$_0$-dilation
of a two-parameter CP$_0$-semigroup which is ``simple" in some sense (e.g. - acts on $M_n(\mathbb{C})$).
\end{enumerate}

In this note, we obtain a partial positive result related to the first problem and a partial negative result related to the second one.
Referring to the notation of Theorem \ref{thm:scedil}, we show that if $\phi$ is not an automorphism semigroup then $\alpha$ is cocycle conjugate to the minimal E$_0$-dilation of $\phi$, and that if $\phi$ is an automorphism semigroup then $\alpha$ is also an automorphism semigroup (and in this case it is cocycle conjugate to $\phi$, which is its own minimal dilation, if and only if $H$ is infinite dimensional). In particular, we conclude that if $\phi$ is not an automorphism semigroup and has a bounded generator (in particular, if $H$ is finite dimensional) then $\alpha$ is a type I E$_0$-semigroup. Needless to say, the same results hold with $\phi$ and $\alpha$ replaced by $\theta$ and $\beta$, respectively.
\begin{remark}\emph{
We emphasize that all this is true when $\alpha$ is the E$_0$-semigroup constructed in the proof of
Theorem \ref{thm:scedil} (\cite[Theorem 6.6]{Shalit07a}). It is not expected that the E$_0$-dilation
of a two-parameter CP$_0$-semigroup be unique (even under a minimality assumption) thus we
state explicitly that all our conclusions are true only for this particular dilation.}
\end{remark}

We are still very far from solving the two problems mentioned above. The first problem is not solved because it is not clear whether the cocycle conjugacy classes of $\alpha$ and $\beta$ determine in any reasonable way the two-dimensional dynamic behavior of the the E$_0$-semigroup $\{\alpha_s \circ \beta_t \}_{s,t \geq 0}$. Let us be a little more concrete in what we mean by this. One may attempt to define the notion of \emph{cocycle equivalence} of two-parameter E$_0$-semigroups exactly as it was defined for one-parameter semigroups, the only difference being that cocycles are now
\emph{two-parameter} families of unitaries. Now assume that $\alpha, \beta$ and $\alpha', \beta'$ are two pairs of commuting E$_0$-semigroups such that $\alpha$ and $\beta$ are cocycle conjugate to $\alpha'$ and $\beta'$, respectively.
In this situation, it is not clear whether the two-parameter semigroups $\{\alpha_s \circ \beta_t\}_{s,t\geq}$
and $\{\alpha'_s \circ \beta'_t\}_{s,t\geq}$ are cocycle conjugate.

The second problem is not solved because we have not ruled out the possibility that for some $a,b > 0$, the one-parameter E$_0$-semigroup $\gamma = \{\gamma_t\}_{t \geq 0}$ given by
$$\gamma_t := \alpha_{at} \circ \beta_{bt}$$
is one that has not been seen before.

\begin{remark}\emph{
This note is a sequel to \cite{Shalit07a}, and the results here depend on the constructions made there. To avoid many
repetitions, we shall refer the reader to that paper for many definitions, constructions and results, as well as for the preliminaries.}
\end{remark}

\section{The simplest case}
Perhaps the simplest kind of two-parameter CP$_0$-semigroups arise as semigroups on $B(H_1 \otimes H_2)$ of the form
\be\label{eq:product}
\psi_{(s,t)} = \phi_s \otimes \theta_t ,
\ee
where $\phi$ is a CP$_0$-semigroup on $B(H_1)$ and $\theta$ is a CP$_0$-semigroup on $B(H_2)$.  It is almost immediate from the definitions that $\psi_{(s,0)} = \phi_s \otimes \textrm{\bf id} $ and $ \psi_{(0,t)} = \textrm{\bf id} \otimes \theta_t$ commute strongly for all $s,t \geq 0$. However, we do not need to appeal to Theorem \ref{thm:scedil} to construct a minimal dilation of $\psi$. If $\alpha$ is the minimal E$_0$-dilation of $\phi$ (acting on $B(K_1)$), and $\beta$ is the minimal E$_0$-dilation of $\theta$ (acting on $B(K_2)$), then the semigroup $\gamma$ acting on $B(K_1 \otimes K_2)$ and given by
\be\label{eq:product2}
\gamma_{(s,t)} = \alpha_s \otimes \beta_t
\ee
is a minimal E$_0$-dilation of $\psi$. Indeed, for all $A \in B(K_1), B \in B(K_2)$,
\begin{align*}
P_{H_1 \otimes H_2} \gamma_{(s,t)} (A \otimes B) P_{H_1 \otimes H_2} &= P_{H_1} \alpha_s (A) P_{H_1} \otimes P_{H_2} \beta_t (B) P_{H_2} \\
&= \phi_s(P_{H_1} A P_{H_1}) \otimes \theta_t (P_{H_2} B P_{H_2}) \\
&= \psi_{(s,t)} \left(P_{H_1 \otimes H_2} (A \otimes B) P_{H_1 \otimes H_2} \right),
\end{align*}
because $P_{H_1 \otimes H_2} = P_{H_1} \otimes P_{H_2}$. To prove that $\gamma$ is a minimal dilation of
$\psi$, we have to show that central support of $P_{H_1 \otimes H_2}$ in $B(K_1 \otimes K_2)$ is $1$, and that
$$W^* \left(\bigcup_{s,t \geq 0} \gamma_{(s,t)}\left(B(H_1 \otimes H_2)\right)\right) = B(K_1 \otimes K_2) .$$
The latter follows from the equalities $B(K_1) = W^*\left(\bigcup_{t\geq 0} \alpha_t(B(H_1))\right)$ and $B(K_2) = W^*\left(\bigcup_{t\geq 0} \beta_t(B(H_2))\right)$, while the former is obvious.

We note that the above discussion works for CP-semigroups $\phi$ and $\theta$ acting on von Neumann algebras $\cM_1$ and $\cM_2$. The only issue that has to be addressed is that of minimality: using \cite[Corollary III.1.5.8]{Blackadar}, (which states that if the central
support of $P_{H_1}$ in $\cR_1$ is $1_{K_1}$ and the central support of $P_{H_2}$ in $\cR_2$ is $1_{K_2}$, then the central support of $P_{H_1} \otimes P_{H_2} = P_{H_1 \otimes H_2}$ in $\cR_1 \otimes \cR_2$ is $1_{K_1} \otimes 1_{K_2} = 1$), one may show that if $(\alpha,\cR_1,K_1)$ and $(\beta,\cR_2,K_2)$ are the minimal dilations
of $(\phi,\cM_1)$ and $(\theta,\cM_2)$, respectively, then $\gamma$ of
(\ref{eq:product2})
is the minimal dilation of $\psi$.

Of course, not all strongly commuting two-parameter CP$_0$-semigroups have the form (\ref{eq:product}) - this can be seen by considering two nontrivial commuting CP$_0$-semigroups on $M_n (\mathbb{C})$ with $n$ prime. However, we will see below that for general strongly commuting CP$_0$-semigroups, the E$_0$-dilation given by Theorem \ref{thm:scedil} is
also ``made up from" the minimal dilations.

\section{Restricting an isometric dilation to a minimal isometric dilation}\label{sec:restricting}

Let $\cS$ be a semigroup, let $X = \{X(s)\}_{s\in\cS}$ be a product system over $\cS$ and let $T$ be a completely contractive covariant representation of $X$ on a Hilbert space $H$. Let $V$ be an isometric dilation of $T$ on a Hilbert space $K \supseteq H$. Define
$$L = \bigvee_{s\in\cS} V_s (X(s)) H .$$
For all $s\in \cS$ and $x \in X(s)$, $L$ is invariant under $V_s(x)$. As $T_0$ is assumed to be nondegenerate, $H \subseteq L$.
We define define a map $W_s: X(s) \rightarrow B(L)$ by
$$W_s(x) = V_s(x) \big|_L .$$
$W = \{W_s\}_{s\in\cS}$ is a representation of $X$ on $L$. Indeed, if $s,t \in \cS$, $x \in X(s), y \in X(t)$ and $l \in L$, then
\begin{align*}
W_{s+t}(x \otimes y)l &= V_{s+t}(x \otimes y)l \\
&= V_s(x) V_t(y) l \\
&= W_s(x) W_t(y) l.
\end{align*}
Clearly, $W$ has the same continuity properties as $V$. In particular, if
$X$ is a product system of W$^*$-correspondences and $V$ is a representation of W$^*$-correspondences (i.e. - $V_s$ is continuous with respect to the $\sigma$-topology on $X(s)$ and the $\sigma$-weak operator topology on $B(K)$), then so is $W$. To see that $W$ is isometric, we first compute $\widetilde{W}$. For $s\in \cS$ and $x \in X(s)$ and $l \in L$ we have
$$\widetilde{W}_s(x \otimes l) = W_s(x) l = V_s(x) l = \widetilde{V}_s(x \otimes l),$$
thus $\widetilde{W}_s = \widetilde{V}_s \big|_{X(s)\otimes L}$. Thus
$$\widetilde{W}_s^* \widetilde{W}_s = P_{X(s)\otimes L} \widetilde{V}_s^* \widetilde{V}_s \big|_{X(s)\otimes L} = I_{X(s)\otimes L} .$$
Most importantly for us, $W$ is also a dilation of $T$: if $s \in \cS$, $x \in X(s)$ and $h \in H$, then
\begin{align*}
P_H W_s(x) h &= P_H V_s(x)\big|_{L} h \\
&= T_s(x) h .
\end{align*}
It is obvious that $W$ is a minimal dilation of $T$, because
$$L = \bigvee_{s\in\cS} V_s (X(s)) H = \bigvee_{s\in\cS} W_s (X(s)) H .$$

\begin{definition}
$W$ is called the \emph{restriction of $V$  to a minimal isometric dilation of $T$}.
\end{definition}

The discussion establishes the following theorem:
\begin{theorem}\label{thm:restrict0}
Let $\cS$ be a semigroup, let $X = \{X(s)\}_{s\in\cS}$ be a product system over $\cS$ and let $T$ be a c.c. representation of $X$. Every isometric dilation of $T$ can be restricted to a minimal
isometric dilation of $T$.
\end{theorem}

For our purposes below, we need a specialization of the above theorem:
\begin{theorem}\label{thm:restrict}
Let $X = \{X(t)\}_{t\geq0}$ be a product system of W$^*$-correspondences over $\Rp$ and let $T$ be a fully-coisometric c.c. representation of $X$ on $H$. Every isometric dilation of $T$ can be restricted to a minimal
isometric and fully-coisometric dilation of $T$.
\end{theorem}

\begin{proof}
All we have to do is to show that the restriction of any isometric dilation of $T$ to a minimal one is fully-coisometric. By a standard computation the minimal isometric dilation of $T$ is unique, up to unitary equivalence. By \cite[Theorem 3.7]{MS02}, the minimal isometric dilation of $T$ is fully-coisometric.
\end{proof}

\section{The type of dynamics that arise in a two-parameter dilation}

Let us fix notation for this section. $H$ is a separable Hilbert space, $\phi$ and $\theta$ are strongly commuting CP$_0$-semigroups on $B(H)$. $K$, $\alpha$ and $\beta$ are as in Theorem \ref{thm:scedil}, and we emphasize again that they are assumed to be given by the construction in the proof of that theorem. Our results below will be stated with assumptions on $\phi$ and conclusions on $\alpha$, but, of course, these results also hold with $\theta$ and $\beta$ instead of $\phi$ and $\alpha$.

We recall how the dilation of $\phi$ and $\theta$ is constructed.
By the constructions in \cite[Section 3]{MS02}, there are product systems of Hilbert spaces $E = \{E(t)\}_{t\geq 0}$ and $F = \{F(t)\}_{t\geq 0}$ and fully-coisometric product system representations $T^E:E \rightarrow B(H)$ and $T^F:F \rightarrow B(H)$
such that
$$\phi_t (A) = \widetilde{T_t^E} (I \otimes A)\widetilde{T_t^E}^* ,$$
and
$$\theta_t (A) = \widetilde{T_t^F} (I \otimes A)\widetilde{T_t^F}^* ,$$
for all $t \geq 0$ and all $A \in B(H)$. By the constructions in \cite[Section 4]{Shalit07a}, we may
form a product system $X$ over $\Rpt$ and a fully-coisometric representation $T: X \rightarrow B(H)$ by
$$X(s,t) = E(s) \otimes F(t)$$
and
$$T_{(s,t)}(x \otimes y) = T_s^E(x) T_t^F(y) \,\, , \,\, x \in E(s), y \in F(t) .$$
By \cite[Theorem 5.2]{Shalit07a}, there is a Hilbert space $K \supseteq H$ and an isometric and fully-coisometric
representation $V: X \rightarrow B(K)$ such that $V$ is a minimal dilation of $T$.
The dilating E$_0$-semigroups $\alpha$ and $\beta$ are given by
$$\alpha_t(A) = \widetilde{V^E_t}(I \otimes A)\widetilde{V^E_t}^* \,\, , \,\, A \in B(K)$$
and
$$\beta_t(A) = \widetilde{V^F_t}(I \otimes A)\widetilde{V^F_t}^* \,\, , \,\, A \in B(K) ,$$
where $V^E$ is the representation of $E$ given by
$$V^E_t(x) = V_{(t,0)}(x \otimes 1) \,\, , \,\, x \in E(t) ,$$
and $V^F$ is the representation of $F$ given by
$$V^F_t(y) = V_{(0,t)}(1 \otimes y) \,\, , \,\, y \in F(t) .$$

\begin{theorem}\label{thm:aut}
If $\phi$ is a semigroup of automorphisms, then so is $\alpha$.
\end{theorem}
\begin{proof}
If $\phi$ is a semigroup of automorphisms, then $E$ turns out to be the trivial bundle $\Rp \times \mathbb{C}$.
In this situation, an isometric and fully-coisometric representation of $E$ is just a semigroup of unitaries. As the formula for $\alpha_t$ shows that it is given by conjugation with a unitary, $\alpha_t$ is an automorphism, for all $t\geq 0$.
\end{proof}

Before proceeding, we write down three (probably well known) facts that we shall need.
\begin{proposition}\label{prop:infdil1}
Let $E$ be a product system of Hilbert spaces over $\Rp$, and let $T$ be a representation of $T$ on $H$.
Let $V$ be the minimal isometric of $T$, representing $E$ on a Hilbert space $G\supseteq H$. If $T$ is not isometric, then $G$ is infinite dimensional.
\end{proposition}
\begin{proof}
Any dilation of the product system representation $T$ contains the minimal dilation of the single c.c. representation $T_t$ of the correspondence $E(t)$, for all $t$. Thus it is enough to show that the minimal isometric dilation of a
single completely contractive covariant representation that is not isometric represents the correspondence on an
infinite dimensional space. This can be dug out of the proof of \cite[Theorem 2.18]{MS02}.
\end{proof}

\begin{proposition}\label{prop:infdil2}
Assume that $\phi$'s minimal E$_0$-dilation
acts on $B(G)$, where $G\supseteq H$ is a Hilbert space. If $\phi$ is not an E$_0$-semigroup itself, then
$G$ is infinite dimensional.
\end{proposition}
\begin{proof}
This follows from the previous proposition and from the uniqueness of the minimal E$_0$-dilation, together
with Muhly and Solel's construction of the minimal E$_0$-dilation in terms of product system representations and
isometric dilations.
\end{proof}
\begin{proposition}\label{prop:Arv}
Let $\gamma$ be an E$_0$-semigroup acting on a separable Hilbert space $G$. Let $P$ be an infinite dimensional
projection in $B(G)$ such that $\gamma_t(P) = P$ for all $t\geq 0$. Let $\sigma$ denote the restriction of $\gamma$ to the invariant corner $PB(G)P = B(PG)$. Then $\sigma$ and $\gamma$ are cocycle conjugate.
\end{proposition}
\begin{proof}
\cite[Proposition 2.2.3]{Arv03}.
\end{proof}

\begin{theorem}\label{thm:cocycle}
If $\phi$ is not a semigroup of automorphisms, then $\alpha$ is cocycle conjugate to $\phi$'s minimal
dilation.
\end{theorem}
\begin{proof}
As in Section \ref{sec:restricting}, let $W$ denote the restriction of $V^E$ to the minimal isometric (and fully-coisometric) dilation of $T^E$, and denote by $L$ the space
on which it represents $E$. By Proposition \ref{prop:infdil1}, ${\rm dim} L = \infty$.
We compute:
\begin{align*}
\alpha_t(P_L) &= \widetilde{V^E_t}(I \otimes P_L)\widetilde{V^E_t}^* \\
&= \widetilde{W}_t(I \otimes P_L)\widetilde{W}_t^* P_L = P_L.
\end{align*}
Let $\sigma$ denote the restriction of $\alpha$ to $B(P_L K)$. By Proposition \ref{prop:Arv}, $\alpha$ and $\sigma$ are cocycle conjugate. It remains to show that $\sigma$ is the minimal dilation of $\phi$. But for all $A \in B(L), t\geq 0$,
\begin{align*}
\sigma_t(A) &= \sigma_t(P_L A P_L) \\
&= \alpha_t(P_L A P_L) \\
&= \widetilde{V^E_t}(I \otimes P_L A P_L)\widetilde{V^E_t}^* \\
&= \widetilde{W}_t(I \otimes A)\widetilde{W}_t^* .
\end{align*}
But $W$ is $T^E$'s minimal dilation. The results of \cite{MS02} show that $\sigma$ must therefore be the minimal E$_0$-dilation of $\phi$.
\end{proof}

\begin{corollary}
$\alpha$ is cocycle conjugate to the minimal dilation of $\phi$ in all cases except the case where
$\phi$ is an automorphism semigroup, $\theta$ is not an automorphism semigroup and $H$ is finite dimensional.
\end{corollary}
\begin{proof}
Assume that $\phi$ is a semigroup of automorphisms. In this case it is, of course, its own minimal dilation. We know by Theorem \ref{thm:aut} that $\alpha$ is also a semigroups of automorphisms. If $H$ is infinite dimensional, then $\alpha$ and $\phi$ are cocycle conjugate (this is the content of Remark 2.2.4, \cite{Arv03}).

Assume further that $H$ is finite dimensional. If $\theta$ is also an automorphism semigroup, then $\alpha = \phi$ (and $\beta = \theta$). Finally, if $\theta$ is not a semigroup of automorphisms, then, by Proposition \ref{prop:infdil2}, $K$ must be infinite dimensional, so $\alpha$ cannot be cocycle conjugate to $\phi$.
\end{proof}

\begin{corollary}
Assume that $\phi$ is not an automorphism semigroup and has a bounded generator. Then $\alpha$ is a type I E$_0$-semigroup.
\end{corollary}
\begin{proof}
This follows from Theorems \ref{thm:Arv} and \ref{thm:cocycle}.
\end{proof}

\begin{remark}
\emph{By the results in \cite{Arv99}, one may also effectively compute the index of $\alpha$ in terms of natural structures associated with the generator of $\phi$.}
\end{remark}

\section{Acknowledgement}
This research is part of the authors PhD. thesis, done under
the supervision of Baruch Solel.



\begin{thebibliography}{9}

\bibitem{Arv99} Wm. B. Arveson, \emph{On the index and dilations of completely positive semigroups}, Internat. J. Math.
\textbf{10} (1999), no.~7, 791--823.


\bibitem{Arv03}Wm. B. Arveson, \emph{Non commutative dynamics and
$E$-semigroups}, Springer Monographs in Math., Springer-Verlag,
2003.

\bibitem{Bhat1996}
B. V. R. Bhat, \emph{An index theory for quantum dynamical semigroups}, Trans.
  Amer. Math. Soc. \textbf{348} (1996), no.~2, 561--583.

%

\bibitem{Blackadar} B. Blackadar, \emph{Operator Algebras:
Theory of C*-Algebras and von Neumann Algebras},
Series: Encyclopedia of Mathematical Sciences , Vol. 122, Springer, 2006.

%
%

%

%

\bibitem {MS02}P. Muhly and B. Solel, \emph{Quantum Markov Processes
(Correspondences and Dilations)}, Internat. J. Math. \textbf{13} (2002), 863--906.
%
%
%
%
%

\bibitem{Pow99} R. T. Powers, \emph{Induction of semigroups of endomorphisms of ${B}({H})$ from completely
	positive semigroups of $(n\times n)$ matrix algebras}, Internat. J. Math.
\textbf{10} (1999), no.~7, 773--790.


%

%
%
%

\bibitem {Shalit07a} O. M. Shalit, \emph{E$_0$-dilation of strongly commuting CP$_0$-semigroups}, preprint, 	 arXiv:0707.1760v2 [math.OA] ,2007.
%
%
%


\end{thebibliography}
\end{document}